\documentclass[12pt,reqno]{amsart}
\usepackage{amsmath, amsfonts, amssymb, amsthm}
\def\Z{\mathbb Z}

\def\F{\mathbb F}

\def\cA{\mathcal A}
\def\cB{\mathcal B}
\def\cS{\mathcal S}
\def\cT{\mathcal T}

\def\fJ{\mathfrak J}
\def\fT{\mathfrak T}
\def\fX{\mathfrak X}

\def\splus#1{\stackrel{#1}+}
\def\1{{\bf 1}}

\theoremstyle{plain}
\newtheorem{theorem}{Theorem}

\newtheorem{proposition}{Proposition}
\newtheorem{lemma}{Lemma}

\theoremstyle{definition}

\theoremstyle{remark}

\begin{document}
\title{On the generalized restricted sumsets in abelian groups}
\author{Shanshan Du}
\email{ssdu@jit.edu.com}
\author{Hao Pan}
\email{haopan79@zoho.com}
\address{ The Fundamental Division, Jingling Institute of Technology,
Nanjing 211169, People's Republic of China}
\address{Department of Mathematics, Nanjing University,
Nanjing 210093, People's Republic of China}
\begin{abstract}
Suppose that $A$, $B$ and $S$ are non-empty subsets of a finite abelian group $G$.
Then the generalized restricted sumset
$$
A\splus{S}B:=\{a+b:\,a\in A,\ b\in B,\ a-b\not\in S\}
$$
contains at least
$$
\min\{|A|+|B|-3|S|,p(G)\}
$$
elements, where $p(G)$ is the least prime factor of $|G|$.
Further, we also have
$$
|A\splus{S}B|\geq \min\{|A|+|B|-|S|-2,p(G)\},
$$
provided that both $|A|$ and $|B|$ are large with respect to $|S|$.
\end{abstract}
\maketitle

\section{Introduction}
\setcounter{lemma}{0}\setcounter{theorem}{0}\setcounter{corollary}{0}
\setcounter{equation}{0}

Suppose that $G$ is a finite abelian group and $A,B$ are two non-empty subsets of $G$. Define the sumset
$$
A+B:=\{a+b:\,a\in A,\ b\in B\}.
$$
For a positive integer $n$, let $\Z_n$ denote the cyclic group of order $n$. If $p$ is prime and $\emptyset\neq A,B\subseteq\Z_p$, then classical Cauchy-Davenport theorem (cf. \cite[Theorem 2.2]{Na96}) says that
\begin{equation}
|A+B|\geq\min\{|A|+|B|-1,p\}.
\end{equation}
For a finite abelian group $G$, let $p(G)$ denote the least prime factor of $|G|$. In view of the well-known Kneser theorem, we have the following extension of the Cauchy-Davenport theorem in abelian groups:
\begin{equation}
|A+B|\geq\min\{|A|+|B|-1,p(G)\},
\end{equation}
where $A,B$ are two non-empty subsets of $G$.

On the other hand, Erd\H os and Heilbronn \cite{Er63,EH64} considered the restricted sumset
$$
A\dotplus B:=\{a+b:\,a\in A,\ b\in B,\ a\neq b\}.
$$
They conjectured that if $p$ is prime and $A,B$ are two non-empty subsets $\Z_p$, then 
\begin{equation}
|A\dotplus A|\geq\min\{2|A|-3,p\}.
\end{equation}
With help of the exterior algebra, Dias da Silva and Hamidoune \cite{DH94} confirmed the conjecture of Erd\H os and Heilbronn.
Subsequently, 
using the polynomial method, Alon, Nathanson and Ruzsa \cite{ANR95,ANR96} gave a simple proof of the Erd\H os-Heilbronn conjecture.
In fact, they obtained that
\begin{equation}
|A\dotplus B|\geq\min\{|A|+|B|-2,p\},
\end{equation}
provided $|A|\neq |B|$.

In \cite{Ka03,Ka04}, K\'arolyi considered the Erd\H os-Heilbronn problem in abelian groups. He proved that
\begin{equation}\label{Karolyi}
|A\dotplus A|\geq\min\{2|A|-3,p(G)\},
\end{equation}
where $A$ is a non-empty subset of a finite abelian group $G$. In \cite{BW09}, Balister and Wheeler showed that 
for $\emptyset\neq A,B\subseteq G$,
\begin{equation}\label{B-W}
|A\dotplus B|\geq\min\{|A|+|B|-3,p(G)\}.
\end{equation}
They even proved that (\ref{B-W}) is still valid when $G$ is a finite group (not necessarily commutative).
The key ingredient of the proofs of K\'arolyi and Balister-Wheeler is an inductive step, i.e., to prove (\ref{Karolyi}) and (\ref{B-W}) under the hypothesis that they are valid for $H$ and $G/H$, where $H$ is a subgroup of $G$.

In this paper, we shall consider the generalized restricted sumset
$$
A\splus{S} B:=\{a+b:\,a\in A,\ b\in B,\ a-b\not\in S\},
$$
where $A,B,S$ are the non-empty subsets of $G$. When $G=\Z_p$ with $p$ is prime, Pan and Sun proved that
\begin{equation}\label{P-S}
|A\splus{S} B|\geq\min\{|A|+|B|-|S|-2,p\},
\end{equation}
provided that $|S|<p$. Notice that the restriction $|S|<p$ is necessary, since if $S=\Z_p$, then
$A\splus{S}B$ is always empty.

It is natural to find a generalization of (\ref{P-S}) for abelian groups. However, as we shall see later, the inductive step will become much more complicated
when $|S|$ is large. Here we can establish the following weak type extension of (\ref{P-S}) for abelian groups.
\begin{theorem}\label{MainT1}
Let $G$ be a finite abelian group. Suppose that $A$, $B$ and $S$ are non-empty subsets of $G$. Then
\begin{equation}\label{ABS3S}
|A\splus{S} B|\geq\min\{|A|+|B|-3|S|,p(G)\}.
\end{equation}
\end{theorem}
In fact, essentially our proof of Theorem \ref{MainT1} doesn't depend on the fact that $G$ is abelian.
In Section 2, we shall also give a brief explanation how to extend (\ref{ABS3S}) to general finite groups.

Although Theorem \ref{MainT1} holds unconditionally, we can get the better lower bound of $|A\splus{S}B|$ under some additional assumptions.
For example, when $G=\Z_{p^\alpha}$, (\ref{ABS3S})  can be improved to \cite[Remark 1.3]{PS06}
\begin{equation}\label{LBAB2S1}|A\splus{S}B|\geq\min\{|A|+|B|-2|S|-1,p\}.\end{equation}
On the other hand, if $|A|,|B|$ are large with respect to $|S|$, we can show that (\ref{P-S}) is also valid for any finite abelian groups.
\begin{theorem}\label{MainT2}
Let $A$, $B$ and $S$ be non-empty subsets of an abelian group $G$. 
Suppose that
$$
\min\{|A|,|B|\}\geq 9|S|^2-5|S|-3.
$$
Then
\begin{equation}\label{ABS1S2}
|A\splus{S} B|\geq\min\{|A|+|B|-|S|-2,p(G)\}.
\end{equation}
\end{theorem}

\section{Proof of Theorem \ref{MainT1}}
\setcounter{lemma}{0}\setcounter{theorem}{0}\setcounter{corollary}{0}
\setcounter{equation}{0}

\begin{lemma}\label{ABSFpa} Let $p\geq 3$ be a prime and $\alpha\geq 1$.
Suppose that $A,B,S$ are non-empty subsets of $\F_{p^{\alpha}}$ with $|S|<p$, where $\F_{p^{\alpha}}$
is the finite field with $p^\alpha$ elements. Then for any $\gamma\in\F_{p^{\alpha}}\setminus\{0,-1\}$, the cardinality of the restricted sumset
$$
\{a+b:\,a\in A,\ b\in B,\ a-\gamma b\not\in S\}
$$
is at least
$$
\min\{|A|+|B|-|S|-2,p\}.
$$
\end{lemma}
\begin{proof}
This is \cite[Corollary 2]{PS02}.
\end{proof}
\begin{lemma}\label{ABSpa2} Let $p$ be a prime and $\alpha\geq 1$.
Suppose that $A=\{a_1,\ldots,a_m\}$, $B=\{b_1,\ldots,b_n\}$ and $S$ are non-empty subsets of $\F_{p^{\alpha}}$ with $|S|<p$.
Let $h=|S|$ and suppose that $m\geq h+3$. If $m+n-h-2\leq p$,then the set
$$
(A\splus{S}B)\setminus\{a_1+b_1,\ldots,a_1+b_n\}
$$
contains the distinct elements
$$
a_{i_1}+b_{j_1},a_{i_2}+b_{j_2},\ldots,a_{i_{m-h-2}}+b_{j_{m-h-2}}
$$
such that for each $1\leq k\leq m-h-2$
$$
i_k\in\{2,3,\ldots,h+2,k+h+2\}.
$$
\end{lemma}
\begin{proof} Let 
$$
X_k=(\{a_1,a_2,\ldots,a_{h+2},a_{k+h+2}\}\splus{S}B)\setminus\{a_1+b_1,\ldots,a_1+b_n\}
$$
for each $1\leq k\leq m-h-2$.
Clearly for any $\emptyset\neq k\subseteq\{1,2,\ldots,m-h-2\}$, using Lemma \ref{ABSFpa}, we have
\begin{align*}
\bigcup_{k\in I}X_k=&\bigg(\bigg(\{a_1,a_2,\ldots,a_{h+2}\}\cup\bigcup_{k\in I}\{a_{k+h+2}\}\bigg)\splus{S}B\bigg)\setminus\{a_1+b_1,\ldots,a_1+b_n\}\\
\leq&\big(h+2+|I|)+|B|-|S|-2\big)-n=|I|.
\end{align*}
With help of the well-known Hall theorem, we may choose distinct $x_1,x_2,\ldots,x_{m-h-2}$ such that $x_k\in X_k$ for each $k=1,\ldots,m-h-2$, i.e., $x_k=a_{i_k}+b_{j_k}$ with $i_k\in\{2,\ldots,h+2,k+h+2\}$.
\end{proof}

\begin{proof}[Proof of Theorem \ref{MainT1}]
Assume that our assertion is true for any proper subgroup of $G$.
Suppose that $|A|+|B|-3|S|>p(G)$. Then we may choose non-empty $A'\subseteq A$ and 
$B'\subseteq B$ such that $|A'|+|B'|-3|S|=p(G)$. Clearly
$|A\splus{S}B|\geq|A'\splus{S}B'|$.
So we may assume that $|A|+|B|-3|S|\leq p(G)$. On the other hand, trivially we have
$$
|A\splus{S}B|\geq\max\{|A|-|S|,\ |B|-|S|\}.
$$
So 
$
|A|+|B|-3|S|>|B|-|S|
$
will imply $|A|>2|S|$, i.e.,  $$|A\splus{S}B|\geq|A|-|S|\geq |S|+1.$$ Hence we always assume that $|S|<p(G)$.

Let $H$ be a proper subgroup of $G$ such that $[G:H]=p(G)$.
Write
$$
A=\bigcup_{i=1}^m(a_i+\cA_i),\qquad B=\bigcup_{i=1}^n(b_i+\cB_i),\qquad
S=\bigcup_{i=1}^h(s_i+\cS_i),
$$
where $\cA_i,\cB_i,\cS_i$ are non-empty subsets of $H$ and $a_i-a_j,b_i-b_j,s_i-s_j\not\in H$ for any distinct $i,j$.
Noting that
$$
|A\splus{S} B|=|(-B)\splus{(-S)}(-A)|,
$$
without loss of generality, we may assume that
$n\geq m$. Furthermore, assume that
$$
|\cA_1|\geq|\cA_2|\geq\cdots\geq|\cA_m|.
$$

For each $a\in G$, let $\bar{a}$ denote the coset $a+H$. Let
$$\bar{A}=\{\bar{a}_1,\ldots,\bar{a}_m\},\qquad \bar{B}=\{\bar{b}_1,\ldots,\bar{b}_n\},\qquad\bar{S}=\{\bar{s}_1,\ldots,\bar{s}_h\}.$$
In view of Lemma \ref{ABSpa},
$$
|\bar{A}\splus{\bar{S}}\bar{B}|\geq|\bar{A}|+|\bar{B}|-|\bar{S}|-2=m+n-h-2.
$$
Let $1\leq \mu_1<\mu_2<\cdots<\mu_r\leq n$ be all integers such that
$$
\bar{a}_1-\bar{b}_{\mu_1},\bar{a}_1-\bar{b}_{\mu_2},\ldots,\bar{a}_1-\bar{b}_{\mu_r}\in\bar{S}.
$$
Without loss of generality, we may assume that $$\bar{a}_1-\bar{b}_{\mu_k}=\bar{s}_{k}$$ for each $1\leq k\leq r$.
Then by the induction hypothesis,
$$
|(a_1+\cA_1)\splus{s_{k}+\cS_{k}}(b_{\mu_k}+\cB_{\mu_k})|=
|\cA_1\splus{\cS_{k}^*} \cB_{\mu_k}|\geq|\cA_1|+|\cB_{\mu_k}|-3|\cS_{k}|,
$$
where $\cS_{k}^*=(b_{\mu_k}-a_1)+s_{k}+\cS_k$.

Let
$$
\fT=(\bar{A}\splus{\bar{S}}\bar{B})\setminus\{\bar{a}_{1}+\bar{b}_{1},\ldots,\bar{a}_{1}+\bar{b}_{n}\}
$$
and let
$
\tau=|\fT|$.
It follows from Lemma \ref{ABSFpa} that
$$
\tau=|\fT|\geq |\bar{A}\splus{\bar{S}}\bar{B}|-n\geq m-h-2.
$$
We may assume that 
$$
\bar{a}_{\gamma_1}+\bar{b}_{\eta_1},\bar{a}_{\gamma_2}+\bar{b}_{\eta_2}\ldots,\bar{a}_{\gamma_{\tau}}+\bar{b}_{\eta_{\tau}}
$$
are distinct elements of $\fT$,
where $2\leq \gamma_1,\ldots,\gamma_{\tau}\leq m$, $1\leq \eta_1,\ldots,\eta_{\tau}\leq m$ and $$\bar{a}_{\gamma_1}-\bar{b}_{\eta_1},\bar{a}_{\gamma_2}-\bar{b}_{\eta_2}\ldots,\bar{a}_{\gamma_{\tau}}-\bar{b}_{\eta_{\tau}}\not\in\bar{S}.$$ Further, in view of Lemma \ref{ABSpa2}, if $m\geq h+3$, we may assume
$$
\gamma_j\in\{2,3,\ldots,h+2,j+h+2\}
$$
for each $1\leq j\leq m-h-2$. Clearly
\begin{align*}
&\bigg(\bigcup_{k=1}^r\big((a_1+\cA_1)\splus{s_{k}+\cS_{k}}(b_{\mu_k}+\cB_{\mu_k})\big)\bigg)\\
\cup&
\bigg(\bigcup_{\substack{1\leq \nu\leq n\\ \nu\not\in\{\mu_1,\ldots,\mu_r\}}}\big((a_1+\cA_1)+(b_{\nu}+\cB_{\mu})\big)\bigg)\\
\cup&\bigg(\bigcup_{j=1}^{\tau}\big((a_{\gamma_j}+\cA_{\gamma_j})+(b_{\eta_j}+\cB_{\eta_j})\big)\bigg)
\end{align*}
forms a disjoint union of a subset of $A\splus{S}B$.
Therefore
\begin{align}\label{ASBrt}
|A\splus{S}B|\geq&\sum_{k=1}^r|\cA_1\splus{\cS_{k}^*}\cB_{\mu_k}|+
\sum_{\substack{1\leq \nu\leq n\\\nu\not\in\{\mu_1,\ldots,\mu_r\}}}|\cA_1+\cB_\nu|+\sum_{j=1}^{\tau}|\cA_{\gamma_i}+\cB_{\eta_j}|\notag\\
\geq&\sum_{k=1}^r(|\cA_1|+|\cB_{\mu_k}|-3|\cS_{k}|)+
\sum_{\substack{1\leq \nu\leq n\\ \nu\not\in\{\mu_1,\ldots,\mu_r\}}}(|\cA_1|+|\cB_\nu|-1)
+\sum_{j=1}^{\tau}|\cA_{\gamma_j}|\notag\\
\geq&n|\cA_1|+|B|+\sum_{j=1}^{\tau}|\cA_{\gamma_j}|-3\sum_{k=1}^r|\cS_{k}|-(n-r).\end{align}

Evidently
$$
3\sum_{k=1}^r|\cS_{k}|-r\leq 3\sum_{k=1}^h|\cS_{k}|-h=3|S|-h.
$$
Let
$$
\Psi=n(|\cA_1|-1)+|B|+\sum_{j=1}^{\tau}|\cA_{\gamma_j}|.
$$
Clearly (\ref{ABS3S}) is true if we can show that
$$
\Psi\geq |A|+|B|-h.
$$
We need to consider the following four cases:\medskip

\noindent (I) $|\cA_1|=1$. \medskip

Note that now $|\cA_1\splus{\cS_{k}^*}\cB_{\mu_k}|$ is greater than or equal to $|\cB_{\mu_k}|-|\cS_{k}|$, rather than $1+|\cB_{\mu_k}|-3|\cS_{k}|$. In view of (\ref{ASBrt}),
\begin{align*}
|A\splus{S}B|\geq&\sum_{k=1}^r(|\cB_{\mu_k}|-|\cS_{k}|)+
\sum_{\substack{1\leq \nu\leq n\\\nu\not\in\{\mu_1,\ldots,\mu_r\}}}|\cB_\nu|
+\sum_{j=1}^{\tau}1\notag\\
\geq&|B|+(m-h-2)-\sum_{k=1}^r(|\cS_{k}|-1)\geq|A|+|B|-|S|-2.\end{align*}\medskip

\noindent (II) $|\cA_1|\geq 2$ and $n\geq h+4$. \medskip

Suppose that $m\geq h+3$. Recall that for each $j=1,\ldots,m-h-2$,
$\gamma_j\in\{2,h+2,j+h+2\}$, i.e., $\gamma_j\geq j+h+2$. So
$$
\sum_{j=1}^{\tau}|\cA_{\gamma_k}|\geq\sum_{j=1}^{m-h-2}|\cA_{\gamma_j}|\geq 
\sum_{k=h+3}^{m}|\cA_{k}|.
$$
It follows that
\begin{align}
\Psi=&n(|\cA_1|-1)+|B|+\sum_{j=1}^{\tau}|\cA_{\gamma_k}|\notag\\
\geq&(h+2)|\cA_1|-(h+2)+\sum_{k=h+3}^{m}|\cA_{k}|+|B|+(n-h-2)(|\cA_1|-1)\notag\\
\geq&|A|+|B|-h.
\end{align}
And if $m\leq h+2$, we also have
\begin{align*}
\Psi\geq n(|\cA_1|-1)+|B|
\geq m|\cA_1|-m+|B|+(n-m)
\geq|A|+|B|-h.
\end{align*}\medskip

\noindent (III) $|\cA_1|\geq 2$, $n\leq h+3$ and $m\geq h+2$. \medskip

Suppose that
$$|\bar{A}\splus{\bar{S}}\bar{B}|\geq m+n-h-1,$$
i.e., $\tau\geq m-h-1$. Then
\begin{align}\label{PsiABtmh1}
\Psi\geq&m(|\cA_1|-1)+|B|+\sum_{j=1}^{\tau}|\cA_{\gamma_k}|\notag\\
\geq&((m-1)|\cA_1|+2)-m+|B|+(|\cA_{\gamma_1}|+(\tau-1))
\geq|A|+|B|-h.
\end{align}

So we may assume that $|\bar{A}\splus{\bar{S}}\bar{B}|=m+n-h-2$, i.e., $\tau\geq m-h-2$. 
Let
$$
\fJ=\{\bar{a}_j:\,|\cA_j|\leq |\cA_1|-1,\ 1\leq j\leq m\}.
$$
If $|\fJ|\geq 2$, then we get $m|\cA_1|\geq |A|+2$.
Thus
\begin{align*}
\Psi\geq m|\cA_1|+|B|+\tau-m \geq|A|+|B|-h.
\end{align*}
Therefore we may assume that $|\fJ|\leq 1$. 

Assume that there exist $1\leq j\leq m$ and $1\leq k\leq n$ such that 
\begin{equation}\label{ajbkAjA1}
\bar{a}_j+\bar{b}_k\not\in \bar{A}\splus{\bar{S}}\bar{B}\qquad\text{and}\qquad|\cA_j|=|\cA_1|.
\end{equation}
Then we can exchange $a_1+\cA_1$ and $a_j+\cA_j$, i.e., set $\bar{a}_j$ as the new $\bar{a}_1$. 
Since $\bar{a}_1+\bar{b}_k\not\in \bar{A}\splus{\bar{S}}\bar{B}$ now , we get
\begin{align*}
\tau=\big|(\bar{A}\splus{\bar{S}}\bar{B})\setminus\{\bar{a}_1+\bar{b}_1,\ldots,\bar{a}_1+\bar{b}_n\}\big|
\geq(m+n-h-2)-(n-1)=m-h-1.
\end{align*}
By (\ref{PsiABtmh1}), we also have
\begin{align*}
\Psi\geq m(|\cA_1|-1)+|B|+\sum_{j=1}^{\tau}|\cA_{\gamma_j}|\geq |A|+|B|-h.
\end{align*}

We still need to find $j,k$ satisfying the assumption (\ref{ajbkAjA1}).
Clearly we may assume that $|\bar{A}\splus{\bar{S}}\bar{B}|<p(G)$, otherwise we immediately get $|A\splus{S}B|\geq p(G)$.
Suppose that $m+n-1\leq p(G)$. Then
$$
|\bar{A}+\bar{B}|-|\bar{A}\splus{\bar{S}}\bar{B}|\geq (m+n-1)-(m+n-h-2)=h+1.
$$
Since $|\fJ|\leq 1$, we have $|\fJ+\bar{S}|\leq h$. So there exist $1\leq j\leq m$ and $1\leq k\leq n$ such that $\bar{a}_j+\bar{b}_k\not\in \bar{A}\splus{\bar{S}}\bar{B}$ and $\bar{a}_j\not\in\fJ$, i.e., $|\cA_j|=|\cA_1|$.

Suppose that $m+n-1>p(G)$, i.e., $(m-1)+n-1\geq p(G)$. Then $$(\bar{A}\setminus\fJ)+\bar{B}=G/H\supsetneq \bar{A}\splus{\bar{S}}\bar{B}.$$ We also can find
$\bar{a}_j+\bar{b}_k\not\in \bar{A}\splus{\bar{S}}\bar{B}$ with $\bar{a}_j\not\in\fJ$.\medskip

\noindent (IV) $|\cA_1|\geq 2$, $n\leq h+3$ and $m\leq h+1$. \medskip

Suppose that $\tau\geq 1$. From (\ref{PsiABtmh1}), we know that
\begin{align*}
\Psi\geq&\big((m-1)|\cA_1|+2\big)+|B|+\big((\tau-1)+\cA_{\gamma_1}\big)-m\\
\geq&|A|+|B|-h+1.
\end{align*}
And if there exists $1\leq j\leq m$ such that $|\cA_j|\leq |\cA_1|-1$, then we also have
\begin{align*}
\Psi\geq&m|\cA_1|+|B|-m\\
\geq&(|A|+1\big)+|B|-(h+1)
=|A|+|B|-h.
\end{align*}
Furthermore, we get
\begin{align*}
\Psi\geq&n(|\cA_1|-1)+|B|\geq m(|\cA_1|-1)+|B|+(n-m)\geq |A|+|B|-h
\end{align*}
provided that $n\leq h$ or $n>m$.

Below we assume that $$\tau=0\qquad,|\cA_1|=\cdots=|\cA_m|,\qquad n=m=h+1.$$
This is the most difficult part of the proof of Theorem \ref{MainT1}.
Clearly we may set $\bar{a}_1=\bar{a}$ for any $\bar{a}\in\bar{A}$.
Assume that $|\bar{A}\splus{\bar{S}}\bar{B}|<p(G)$.
Note that $\tau=0$ implies
$$
\bar{A}\splus{\bar{S}}\bar{B}\subseteq\{\bar{a}_1+\bar{b}_1,\ldots,\bar{a}_1+\bar{b}_m\},
$$
On the other hand, since $$
|\bar{A}\splus{\bar{S}}\bar{B}|<\min\{2m-1,p(G)\}\leq|\bar{A}+\bar{B}|,$$ 
we may choose $\bar{a}_1=\bar{a}\in\bar{A}$ and $\bar{b}_\mu\in\bar{B}$ such that 
$\bar{a}_1+\bar{b}_\mu\not\in\bar{A}\splus{\bar{S}}\bar{B}$.
So
$$
\bar{A}\splus{\bar{S}}\bar{B}\subseteq\{\bar{a}_1+\bar{b}_k:\,1\leq k\leq m,\ k\neq\mu\}.
$$
But
$$
|\bar{A}\splus{\bar{S}}\bar{B}|\geq 2m-h-2=h=m-1.
$$
We must have
$$
\bar{A}\splus{\bar{S}}\bar{B}=\{\bar{a}_1+\bar{b}_k:\,1\leq k\leq m,\ k\neq\mu\},
$$
i.e., $|\bar{A}\splus{\bar{S}}\bar{B}|=m-1$.

Arbitrarily choose $\bar{a}\in\bar{A}$ as $\bar{a}_1$. 
Assume that $$\bar{a}_1-\bar{b}_{\mu_k}=\bar{s}_{\lambda_k},\qquad 1\leq k\leq r$$ and 
$$\bar{a}_1-\bar{b}_{\nu}\not\in\bar{S},\qquad \nu\not\in\{\mu_1,\ldots,\mu_r\}.$$
Without loss of generality, assume that
$\bar{a}_1+\bar{b}_{\mu_1}\not\in\bar{A}\splus{\bar{S}}\bar{B}$, i.e., $\bar{A}\splus{\bar{S}}\bar{B}=\{\bar{a}_1+\bar{b}_k:\,k\neq\mu_1\}$. 
Letting $\cS_{\lambda_k}^*=(b_k-a_k)+s_{\lambda_k}+\cS_{\lambda_k}$, we get
\begin{align}\label{ASBrm}
|A\splus{S}B|\geq&\sum_{k=1}^r|\cA_1\splus{\cS_{\lambda_k}^*}\cB_{\mu_k}|+
\sum_{\substack{1\leq\nu\leq m\\ \nu\not\in\{\mu_1,\ldots,\mu_r\}}}|\cA_1+\cB_{\nu}|\notag\\
\geq&\sum_{k=1}^r(|\cA_1|+|\cB_{\mu_k}|-3|\cS_{\lambda_k}|)+
\sum_{\substack{1\leq\nu\leq m\\ \nu\not\in\{\mu_1,\ldots,\mu_r\}}}(|\cA_1|+|\cB_\nu|-1)\notag\\
=&m(|\cA_1|-1)+\sum_{k=1}^m|\cB_k|-\sum_{k=1}^r(3|\cS_{\lambda_k}|-1).
\end{align}
Hence
\begin{equation}\label{ASBdelta}
|A\splus{S}B|\geq |A|+|B|-m-3|S|+h+\delta=|A|+|B|-3|S|-1+\delta,
\end{equation}
where $\delta\in\{0,1\}$ and $\delta=1$ if none of the following conditions is true:

\medskip
\noindent (a) $r=h$;

\medskip
\noindent (b) $|\cA_1\splus{\cS_{\lambda_k}^*}\cB_{\mu_k}|=|\cA_1|+|\cB_{\mu_k}|-3|\cS_{\lambda_k}|$ for each $1\leq k\leq h$;

\medskip
\noindent (c) $|\cA_1+\cB_\nu|=|\cA_1|+|\cB_\nu|-1$ for the unique $\nu\not\in\{\mu_1,\ldots,\mu_h\}$.

\medskip
Suppose that there exists some $\bar{a}\in\bar{A}$  such that $\delta=1$ when $\bar{a}_1=\bar{a}$,  i.e., at least one of (a)-(c) fails. Evidently by (\ref{ASBdelta}), we immediately get (\ref{ABS3S}). 
Assume that for any $\bar{a}\in\bar{A}$, we always have $\delta=0$ if $\bar{a}_1=\bar{a}$, i.e., (a)-(c) all hold. 

For each $2\leq i\leq m$, it is impossible that $$\bar{a}_i-\bar{b}_\nu\not\in\bar{S},$$
where $1\leq\nu\leq m$ is the unique one satisfying $\bar{a}_1-\bar{b}_\nu\not\in\bar{S}$.
Otherwise, by (a),
we also can get $$\{\bar{b}_k:\,1\leq k\leq m,\ k\neq\nu\}=\bar{a}_i-\bar{S}.$$ It follows that $$\bar{a}_i-\bar{a}_1+\bar{S}=\bar{S},$$ i.e.,
$\bar{S}=G/H$. This contradicts our assumption $h<m\leq p(G)$.

Below we need to the following inverse theorem of Kar\'olyi \cite[Theorem 4]{Ka05}.
\begin{lemma}\label{inverse} 
Let $A$ and $B$ be non-empty subsets of a finite group $G$.
Suppose that $|A+B|=|A|+|B|-1\leq p(G)-1$. Then the following one holds:

\medskip\noindent(1) $|A|=1$ or $|B|=1$;

\medskip\noindent(2) $A$ and $B$ are two arithmetic progressions with the common difference, i.e.,
$A=\{a,a+q,\ldots,a+(k-1)q\}$ and $B=\{b,q+b,\ldots,(l-1)q+b\}$;

\medskip\noindent(3) $A\subseteq a+K$ and $B\subseteq K+b$, where $K$ is a subgroup of $G$ of order $p(G)$.
\end{lemma}
Since $|\cA_1|+|\cB_\nu|>p(H)$ implies $|A\splus{S}B|\geq|\cA_1+\cB_\nu|\geq p(H)$, we may assume that
\begin{equation}\label{A1BmpH}
|\cA_1|+|\cB_\nu|\leq p(H).
\end{equation}
By (c), we get
\begin{equation}\label{A1Bnu1}
|\cA_1+\cB_{\nu}|=|\cA_1|+|\cB_{\nu}|-1.
\end{equation}
By Lemma \ref{inverse} and (\ref{A1Bnu1}), we must have $|\cB_{\nu}|=1$, or $\cA_1$ and  $\cB_\nu$ are the arithmetic progressions with the common difference, or $\cA_1=\alpha+K$ and  $\cB_\nu=\beta+K$ where $K$ is a subgroup of $H$ of order $p(H)$.

For each $2\leq i\leq m$, since $\bar{a}_i-\bar{b}_\nu\in\bar{S}$,
we may assume that $\bar{a}_i-\bar{b}_\nu=\bar{s}_{\kappa_i}
$ for some $1\leq \kappa_i\leq h$. And there exists a unique $1\leq \upsilon_i\leq m$ such that
$\bar{a}_i+\bar{b}_{\upsilon_i}\not\in \bar{A}\splus{\bar{S}}\bar{B}$.
Note that (a)-(c) are still valid even if we exchange $a_1+\cA_1$ and $a_i+\cA_i$. It follows from (b) and (c) that
\begin{equation}\label{AiBnuSki3}
|\cA_i\splus{\cS_{\kappa_i}^*}\cB_{\nu}|=|\cA_i|+|\cB_{\nu}|-3|\cS_{\kappa_i}^*|
\end{equation}
and
\begin{equation}\label{AiBupi1}
|\cA_i|+\cB_{\upsilon_i}|=|\cA_i|+|\cB_{\upsilon_i}|-1,
\end{equation}
where $\cS_{\kappa_i}^*=s_{\kappa_i}+b_m-a_2+\cS_{\kappa_i}$.

According to Lemma \ref{inverse}, there are four sub-cases:
\medskip

\noindent(i) $|\cB_{\nu}|=1$. \medskip

Trivially $$
|\cA_i\splus{\cS_{\kappa_i}^*}\cB_{\nu}|=|\cA_i|-|\cS_{\kappa_i}|>|\cA_2|+|\cB_{\nu}|-3|\cS_{\kappa_i}|,$$ which evidently contradicts (\ref{AiBnuSki3}).\medskip

\noindent(ii) There exist the subgroups $K_1,K_2$ of $H$ of order $p(H)$ and $\alpha,\beta\in H$ such that
$\cA_{i}\subseteq \alpha+K_1$ and $\cB_{\nu}\subseteq \beta+K_2$.\medskip

Since $|K_1|=|K_2|=p(H)$, either $K_1=K_2$ or $|K_1\cap K_2|=1$. If $|K_1\cap K_2|=1$, it is easy to see that
$$
|\cA_i+\cB_\nu|=|(\cA_i-\alpha)+(\cB_\nu-\beta)|=|\cA_i|\cdot|\cB_\nu|.
$$
Similarly, $|\cA_i-\cB_\nu|=|\cA_i|\cdot|\cB_\nu|$. So $$|\{(a,b):\,a-b\in\cS_{\kappa_i}^*,\ a\in\cA_i,\ b\in\cB_\nu\}|\leq |\cS_{\kappa_i}^*|.$$
Thus
\begin{align*}
|\cA_i\splus{\cS_{\kappa_i}^*}\cB_\nu|\geq&|\cA_i|\cdot|\cB_\nu|-|\cS_{\kappa_i}|\\
\geq&|\cA_i|+|\cB_\nu|-|\cS_{\kappa_i}|-1>
|\cA_i|+|\cB_\nu|-3|\cS_{\kappa_i}|.
\end{align*}

Suppose that $K_1=K_2$. Then letting $\cS_{\kappa_i}^\circ=(\beta-\alpha+\cS_{\kappa_i}^*)\cap K_1$ and applying Lemma \ref{ABSFpa} , we get
$$
|\cA_i\splus{\cS_{\kappa_i}^*}\cB_\nu|=|(\cA_2-\alpha)\splus{\cS_{\kappa_i}^\circ}(\cB_m-\beta)|\geq |\cA_i|+|\cB_\nu|-|\cS_{\kappa_i}^{\circ}|-2.
$$
From (\ref{AiBnuSki3}), we obtain that $|\cS_{\kappa_i}|=1.$\medskip

\noindent(iii) Either $\cA_{i}$ is an arithmetic progression and $\cB_{\nu}\subseteq \beta+K$ where $K$ is a subgroup of $H$, or 
$\cA_{i}=\alpha+K$ and $\cB_{\nu}$ is an arithmetic progression.\medskip

We only need to consider the first possibility, i.e., $\cA_i=\{a,a+q,\ldots,a+(d-1)q\}$ and $\cB_\nu=\beta+K$.
If $q\in K$, then $\cA_i\in a+K$.
By the discussion in  Case (ii), we can get that $|\cS_{\kappa_i}|=1$. 

Suppose that $q\not\in K$. Since $|\cA_i|=d<p(H)$ by (\ref{A1BmpH}), $a+K,\ldots,a+(d-1)q+K$ must be disjoint cosets of $K$. Otherwise, we can get $d\geq [\langle q,K\rangle:K]\geq p(H)$,
where $\langle q,K\rangle$ denotes the subgroup generated by $q$ and $K$. Write
$$
\cS_{\kappa_i}^*=\bigcup_{j=1}^{h^*}(t_j+\cT_j),
$$
where $\cT_j\subseteq K$ and $t_j-t_k\not\in K$ if $j\neq k$.
Assume that $a+\mu_k^*q-t_{k}\in K$ for $1\leq k\leq r^*$. Then
\begin{align*}
|\cA_i\splus{\cS_{\kappa_i}^*}\cB_\nu|\geq&\sum_{k=1}^{r^*}(|\cB_\nu|-|\cT_k|)+|B_\nu|\cdot(d-r^*)\\
=&d|\cB_\nu|-\sum_{k=1}^{r^*}|\cT_k|
\geq|\cA_i|+|\cB_\nu|-|\cS_{\kappa_i}|-1.
\end{align*}

\noindent(iv) Both $\cA_{i}$ and $\cB_{\nu}$ are arithmetic progressions.\medskip

Write $$\cA_i=\{a_1+kq_1:\,0\leq k\leq d_1-1\},\qquad\cB_\nu=\{a_2+kq_2:\,0\leq k\leq d_2-1\}.$$
Let $\langle q\rangle$ denote the subgroup generated by $q$. 
If $q_1\not\in\langle q_2\rangle$, then $\cB_\nu\subseteq a_2+\langle q_2\rangle$ and
$a_1+\langle q_2\rangle,\ldots,a_1+(d_1-1)q_1+\langle q_2\rangle$ are  disjoint cosets of $\langle q_2\rangle$.
According to the discussion in Case (iii), we can get $|\cA_i\splus{\cS_{\kappa_i}^*}\cB_\nu|\geq|\cA_i|+|\cB_\nu|-|\cS_{\kappa_i}|-1$. Similarly, the same result can be deduced from the assumption $q_2\not\in\langle q_1\rangle$, too.
 
Suppose that $\langle q_1\rangle=\langle q_2\rangle$. 
Let $K$ be a maximal proper subgroup of $\langle q_1\rangle$, i.e., $[\langle q_1\rangle:K]$ is prime.
Let $$\cA^\circ=\cA_i-a_1,\qquad\cB^\circ=\cB_\nu-a_2,\qquad\cS^\circ=(a_2-a_1+\cS_{\kappa_i}^*)\cap K.$$ Clearly $|\cA_i\splus{\cS_{\kappa_i}^*}\cB_\nu|=|\cA^\circ\splus{\cS^\circ}\cB^\circ|$. Let $$\widehat{\cA}^\circ=\{a+K:\,a\in \cA^\circ\},\quad\widehat{\cB}^\circ=\{b+K:\,b\in \cB^\circ\},\quad\widehat{\cS}^\circ=\{s+K:\,s\in \cS^\circ\}.$$ Since $q_1,q_2\not\in K$, we have $|\widehat{\cA}^\circ|=|{\cA}^\circ|$,
$|\widehat{\cB}^\circ|=|{\cB}^\circ|$ and
$|\widehat{\cS}^\circ|\leq |{\cS}^\circ|$. Since $[\langle q_1\rangle:K]$ is prime, by Lemma \ref{ABSFpa}, we get
\begin{align*}
|\cA_i\splus{\cS_{\kappa_i}^*}\cB_\nu|\geq|\widehat{\cA}^\circ \splus{{\widehat{\cS}^\circ}}\widehat{\cB}^\circ|\geq
|\widehat{\cA}^\circ|+|\widehat{\cB}^\circ|-|\widehat{\cS}^\circ|-2\geq 
|\cA_i|+|\cB_\nu|-|\cS_{\kappa_i}|-2,
\end{align*}
which implies $|\cS_{\kappa_i}|=1$ by (\ref{AiBnuSki3}).\medskip

Now we have deduced that either $$|\cA_i\splus{\cS_{\kappa_i}^*}\cB_\nu|\geq|\cA_i|+|\cB_\nu|-|\cS_{\kappa_i}|-1>|\cA_i|+|\cB_\nu|-3|\cS_{\kappa_i}|$$
which leads to a contradiction to (\ref{AiBnuSki3}), or 
$$|\cA_i\splus{\cS_{\kappa_i}^*}\cB_\nu|=|\cA_i|+|\cB_\nu|-|\cS_{\kappa_i}|-2$$
which implies $|\cS_{\kappa_i}|=1$ from (\ref{AiBnuSki3}).
Since $\bar{a}_i-\bar{b}_\nu\in\bar{S}$ for each $2\leq i\leq m$,
we get $\{\bar{a}_2,\ldots,\bar{a}_m\}=\bar{b}_\nu+\bar{S}$. Hence we must have $$|\cS_{k}|=1$$ for each $1\leq k\leq h$.

If $m=h+1<p(G)$, then
$$
|\bar{A}+\bar{B}|\geq\min\{2m-1,p(G)\}>m.
$$
So there exist
$2\leq j\leq m$ and $1\leq k\leq m$ such that 
$$
\bar{a}_j+\bar{b}_k\in(\bar{A}+\bar{B})\setminus\{\bar{a}_1+\bar{b}_1,\ldots,\bar{a}_1+\bar{b}_m\}.
$$
Assume that $\bar{a}_j-\bar{b}_k=\bar{s}_{\lambda_0}$. Let $\cS_{\lambda_0}^*=s_{\lambda_0}+b_k-a_j+\cS_{\lambda_0}$.
In view of (\ref{ASBrm}), we have
\begin{align*}
|A\splus{S}B|\geq&|\cA_j\splus{\cS_{\lambda_0}^*}\cB_{k}|+\sum_{k=1}^r|\cA_1\splus{\cS_{\lambda_k}^*}\cB_{\mu_k}|+
\sum_{\substack{1\leq\nu\leq m\\ \nu\not\in\{\mu_1,\ldots,\mu_r\}}}|\cA_1+\cB_\nu|\notag\\
\geq&|\cA_j\splus{\cS_{\lambda_0}^*}\cB_{k}|+|A|+|B|-3|S|-1.
\end{align*}
Since $|\cA_j|\geq 2$ and $|\cS_{\lambda_0}^*|=1$, $\cA_j\splus{\cS_{\lambda_0}^*}\cB_{k}\neq\emptyset$, i.e.,
$|A\splus{S}B|\geq |A|+|B|-3|S|$.

Suppose that $m=p(G)$. Then we also have $\cA_1\splus{\cS_{\lambda_1}^*}\cB_{\mu_1}\neq\emptyset$. Recall that $\bar{A}\splus{\bar{S}}\bar{B}=\{\bar{a}_1+\bar{b}_k:\,1\leq k\leq m,\ k\neq\mu_1\}$. We get
\begin{align*}
|A\splus{S}B|\geq&|\cA_1\splus{\cS_{\lambda_1}^*}\cB_{\mu_1}|+(m-1)\geq p(G).
\end{align*}
\end{proof}

Let us briefly explain how to extend Theorem \ref{MainT1} to finite non-commutative groups.
Suppose that $G$ is a finite group and $H$ is a non-trivial subgroup of $G$.
Note that for $a,b\in G$ and $\cA,\cB\subseteq H$, 
$$
(a+\cA)+(b+\cB)=(a+b)+\big(\psi_b(\cA)+\cB\big),
$$
where $\psi_b:\, x\mapsto -b+x+b$ is an inner automorphism of $G$. So we have to study the restricted sumset
$$
A\splus{\sigma,S} B:=\{a+b:\, a\in A,\ b\in B,\ a-\sigma(b)\not\in S\},
$$
where $\sigma$ is an automorphism of $G$. If $a-b\in s+H$ and $\cS\subseteq H$, then we have
$$
(a+\cA)\splus{\sigma,s+\cS}(b+\cB)=
(a+b)+(\psi_b(\cA)\splus{\psi_b\sigma,\cS^*}\cB),
$$
where $\cS^*=\psi_{\sigma(b)+b}\big((\psi_b-a)+s+\cS\big)$.

On the other hand, in order to use the induction process, 
Balister and Wheeler proved the following lemma.
\begin{lemma}[{\cite[Theorem 3.2]{BW09}}]\label{bwl}
Suppose that $G$ is a finite group of odd order and $\sigma$ is an automorphism of $G$. Then there exists
a proper normal subgroup $H$ of $G$ satisfying that

\medskip\noindent (1) $\sigma(H)=H$.

\medskip\noindent (2) $G/H$ is isomorphic to the additive group of some finite field $\F_{p^\alpha}$.

\medskip\noindent (3) Let $\chi$ denote the isomorphism from $G/H$ to the additive group of $\F_{p^\alpha}$.

\medskip\noindent Then there exists some $\gamma\in\F_{p^\alpha}\setminus\{0\}$ such that
$\chi(\sigma(\bar{a}))=\gamma\cdot\chi(\bar{a})$ for each $\bar{a}\in G/H$.
\end{lemma}
With help of Lemmas \ref{ABSFpa}, \ref{inverse} and \ref{bwl}, we can obtain that the following generalization of Theorem \ref{MainT1} for
general finite groups.
\begin{theorem}
Suppose that $A$, $B$ and $S$ are non-empty subsets of a finite group $G$.
Let $\sigma$ be an automorphism of $G$ with odd order.
Then
\begin{equation}
|A\splus{\sigma,S} B|\geq\min\{|A|+|B|-3|S|,p(G)\}.
\end{equation}
\end{theorem}

\section{Proof of Theorem \ref{MainT2} for $\Z_{p^\alpha}$}
\setcounter{lemma}{0}\setcounter{theorem}{0}\setcounter{corollary}{0}
\setcounter{equation}{0}

\begin{lemma}\label{ABSpa} Let $p$ be a prime and $\alpha\geq 1$.
Suppose that $A,B,S$ are non-empty subsets of $\Z_{p^{\alpha}}$. Then
$$
|A\splus{S}B|\geq\min\{|A|+|B|-2|S|-1,p\}.
$$
\end{lemma}
\begin{proof}
See \cite[Remark 1.3]{PS06}.
\end{proof}
\begin{lemma}[{\cite[Lemma 2.1]{DP}}]\label{ABmn1G}
Suppose that $A=\{a_1,\ldots,a_m\}$ and $B=\{b_1,\ldots,b_n\}$ are non-empty subsets of a finite abelian group $G$.
If $m+n-1\leq p(G)$, then the set
$(A+B)\setminus\{a_1+b_1,\ldots,a_1+b_n\}$
contains the distinct elements
$$
a_{2}+b_{j_2},a_{3}+b_{j_3},\ldots,a_{i_{m}}+b_{j_{m}}.
$$
\end{lemma}
\begin{lemma}\label{ABSmn3hG}
Suppose that $A=\{a_1,\ldots,a_m\}$, $B=\{b_1,\ldots,b_n\}$ and $S$ are non-empty subsets of a finite abelian group $G$.
Let $h=|S|$ and suppose that $m\geq 3h+1$. If $m+n-3h\leq p(G)$ and $h<p(G)$, then the set
$
(A\splus{S}B)\setminus\{a_1+b_1,\ldots,a_1+b_n\}
$
contains the distinct elements
$$
a_{i_1}+b_{j_1},a_{i_2}+b_{j_2},\ldots,a_{i_{m-3h}}+b_{j_{m-3h}}
$$
such that for each $1\leq k\leq m-3h$
$$
i_k\in\{2,3,\ldots,3h,k+3h\}.
$$
\end{lemma}
\begin{proof} This lemma immediately follows from Theorem \ref{MainT1} by using the same discussions in the proof of Lemma \ref{ABSpa}.
\end{proof}

\begin{proposition}\label{T2pa}
Let $A,B,S$ be non-empty subsets of $\Z_{p^\alpha}$. Suppose that
$$
\min\{|A|,|B|\}\geq 6|S|^2-5.
$$
Then
\begin{equation}
|A\splus{S}B|\geq \min\{|A|+|B|-|S|-2,p\}.
\end{equation}
\end{proposition}
\begin{proof}

We use an induction on $\alpha$.  Assume that the assertion of Proposition \ref{T2pa} is true for $\Z_{p^{\alpha-1}}$.
In view of (\ref{B-W}), we alway assume that $|S|\geq 2$.
Note that 
$$
|A\splus{S}B|\geq|A|-|S|\geq 6|S|^2-|S|-5\geq |S|
$$
when $|S|\geq 2$. So we only need to consider that case $|S|<p$.

Let $H$ be the subgroup of $\Z_{p^\alpha}$ of order $p$. For $x\in \Z_{p^\alpha}$, let $\bar{x}$ denote the coset $x+H$, and let $\bar{X}:=\{\bar{x}:\,x\in X\}$ for $X\subseteq\Z_{p^\alpha}$. Assume that $$\bar{A}=\{\bar{a}_1,\ldots,\bar{a}_m\},\qquad \bar{B}=\{\bar{b}_1,\ldots,\bar{b}_n\},\qquad\bar{S}=\{\bar{s}_1,\ldots,\bar{s}_h\}.$$
By exchanging $A$ and $B$, we may assume that $m\leq n$.
Write
$$
A=\bigcup_{i=1}^m(a_i+\cA_i),\qquad B=\bigcup_{i=1}^n(b_i+\cB_i),\qquad
S=\bigcup_{i=1}^h(s_i+\cS_i),
$$
where those $\cA_i,\cB_i,\cS_i\subseteq H$.
Moreover, assume that
$$
|\cA_1|\geq|\cA_2|\geq\cdots\geq|\cA_m|,\qquad
|\cB_1|\geq|\cB_2|\geq\cdots\geq|\cB_n|.
$$

Suppose that $m=n=1$. Let $\cT=(b_1-a_1)+s_i+\cS_i$ if $\bar{a}_1-\bar{b}_1=\bar{s}_i$ for some $1\leq i\leq h$,
and let $\cT=\emptyset$ if $\bar{a}_1-\bar{b}_1\not\in\bar{\cS}$. Then
\begin{align*}
|A\splus{S}B|=|\cA_1\splus{\cT}\cB_1|\geq&\min\{|\cA_1|+|\cB_1|-|\cT|-2,p\}\\
\geq&\min\{|A|+|B|-|S|-2,p\}.
\end{align*}
Below we always assume that either $m>1$, or $n>m$.

\medskip\noindent (I) $|\cA_1|\geq 2$.

\medskip
Assume that $\bar{a}_1-\bar{b}_{\mu_k}=\bar{s}_{\lambda_k}$ for each $1\leq k\leq r$, and
$\bar{a}_1-\bar{b}_\nu\not\in\bar{S}$ if $1\leq\nu\leq n$ and $\nu\not\in\{\mu_1,\ldots,\mu_r\}$.
Let $\tau=|(\bar{A}\splus{\bar{S}}\bar{B})\setminus\{\bar{a}_1+\bar{b}_1,\ldots,\bar{a}_1+\bar{b}_n\}|$.
By Lemma \ref{ABSpa}, we have $\tau\geq\max\{m-2h-1,0\}$.
Furthermore, when $m\geq 2h+2$, we may assume that
$
\bar{a}_{\gamma_1}+\bar{b}_{\eta_1},\ldots,\bar{a}_{\gamma_{m-3h}}+\bar{b}_{\eta_{m-2h-1}}
$
are distinct elements of $(\bar{A}\splus{\bar{S}}\bar{B})\setminus(\bar{a}_1+\bar{B})$ with
$\gamma_j\leq j+2h+1$ for each $1\leq j\leq m-2h-1$.

Letting $\cS_{\lambda_k}^*=(b_{\mu_k}-a_k)+s_{\lambda_k}+\cS_{\lambda_k}$, we have
$$
|A\splus{S}B|\geq\sum_{k=1}^r|\cA_1\splus{\cS_{\lambda_k}^*}\cB_{\mu_k}|+
\sum_{\substack{1\leq\nu\leq m\\ \nu\not\in\{\mu_1,\ldots,\mu_r\}}}|\cA_1+\cB_\nu|+\sum_{j=1}^{\tau}|\cA_{\gamma_j}+\cB_{\eta_j}|.
$$
We may assume that those
$$
|\cA_1\splus{\cS_{\lambda_k}^*}\cB_{\mu_k}|,\ |\cA_1+\cB_\nu|,\ |\cA_{\gamma_j}+\cB_{\eta_j}|<p,
$$
otherwise we immediately get $|A\splus{S}B|\geq p$. Thus
\begin{align}
|A\splus{S}B|\geq&\sum_{k=1}^r(|\cA_1|+|\cB_{\mu_k}|-|\cS_{\lambda_k}|-2)+
\sum_{\substack{1\leq\nu\leq m\\ \nu\not\in\{\mu_1,\ldots,\mu_r\}}}(|\cA_1|+|\cB_\nu|-1)+\sum_{j=1}^{\tau}|\cA_{\gamma_j}|\notag\\
\geq&n|\cA_1|+|B|+\sum_{j=1}^{m-3h}|\cA_{\gamma_j}|-\sum_{k=1}^r(|\cS_{\lambda_k}|+1)-n.\label{A1Bm3hSk}
\end{align}
When $m\geq 2h+2$, we obtain that
\begin{align}\label{ASBpam2h2n2h1}
|A\splus{S}B|\geq&(2h+1)|\cA_1|+\sum_{i=2h+2}^{m}|\cA_{i}|+|B|+(n-2h-1)|\cA_1|-n-\sum_{k=1}^h(|\cS_{k}|+1)\notag\\
\geq&|A|+|B|-|S|-2+\big((n-2h-1)|\cA_1|-n-h+2\big).
\end{align}
While if $m\leq 2h+1$, we also have
\begin{align}\label{ASBpam2h1nm}
|A\splus{S}B|\geq&n|\cA_1|+|B|-\sum_{k=1}^r(|\cS_{\lambda_k}|+1)-n\notag\\
\geq&m|\cA_1|+|B|+(n-m)|\cA_1|-\sum_{k=1}^h(|\cS_{k}|+1)-n\notag\\
\geq&|A|+|B|-|S|-2+\big((n-m)|\cA_1|-n-h+2\big).
\end{align}

If $n\geq5h$, then
$$
(n-2h-1)|\cA_1|-n-h+2\geq 2(n-2h-1)-n-h+2\geq 0.
$$
Below assume that $n\leq 5h-1$. Note that
the function $(x-2h-1)|A|/x-x$ is increasing on $(0,\sqrt{(2h+1)|A|}]$ and is 
decreasing on $[\sqrt{(2h+1)|A|},+\infty)$. Since $m|\cA_1|\geq |A|$,
\begin{align*}
(n-2h-1)|\cA_1|-n\geq&(n-2h-1)\cdot\frac{|A|}{m}-n\\
\geq&
(n-2h-1)\cdot\frac{|A|}{n}-n\\
\geq&\min\bigg\{\frac{|A|}{2h+2}-(2h+2),\ \frac{(3h-2)|A|}{5h-1}-(5h-1)\bigg\}.
\end{align*}

\medskip\noindent (i) $|S|\geq h+1$.\medskip
 
Suppose that $m\geq 2h+2$.
Since $|A|\geq 6|S|^2-5>6h(h+2),$ we have
$$
\frac{|A|}{2h+2}-(2h+2)\geq 3h-(2h-2)=h-2.
$$
And it is easy to check
$$
\frac{(3h-2)|A|}{5h-1}-(5h-1)\geq 
\frac{(3h-2)\cdot 6h(h+2)}{5h-1}-5h+1\geq h-2
$$
for any $h\geq 1$. Hence we have $|A\splus{S}B|\geq|A|+|B|-|S|-2$ when $|S|>h$.

Suppose that $m\leq 2h+1$ and $m<n$.   
Then
\begin{align*}
(n-m)|\cA_1|-n\geq &\frac{n-m}{m}\cdot|A|-n\geq \frac{|A|}{m}-m-1\\
\geq&\frac{|A|}{2h+1}-(2h+1)-1
\geq\frac{6h(h+2)}{2h+1}-2h-2\geq h-2. 
\end{align*}
From (\ref{ASBpam2h1nm}), we also can get $|A\splus{S}B|\geq|A|+|B|-|S|-2$.

Suppose that $m=n\leq 2h+1$.
Since $|\bar{A}+\bar{B}|\geq m+1$, there exist $1\leq \gamma,\eta\leq h$ such that $\bar{a}_\gamma+\bar{b}_\eta\not\in\{\bar{a}_1+\bar{b}_1,\ldots,\bar{a}_1+\bar{b}_m\}$. 
Let $\cT=(b_\eta-a_\gamma)+s_i+\cS_i$ if $\bar{a}_\gamma-\bar{b}_\eta=\bar{s}_i$, and let $\cT=\emptyset$ if $\bar{a}_\gamma-\bar{b}_\eta\not\in\bar{S}$.
Clearly
$$
|\cT|\leq\max_{1\leq i\leq h}|\cS_i|\leq |S|-h+1.
$$
Then
\begin{align*}
|A\splus{S}B|\geq&\sum_{k=1}^r|\cA_1\splus{\cS_{\lambda_k}^*}\cB_{\mu_k}|+
\sum_{\substack{1\leq\nu\leq m\\ \nu\not\in\{\mu_1,\ldots,\mu_r\}}}|\cA_1+\cB_\nu|+|\cA_\gamma\splus{\cT}\cB_\eta|\\
\geq&m|\cA_1|+|B|-|S|-h-n+(|\cA_\gamma|-|\cT|)\\
\geq&|A|+|B|-|S|-2+(|\cA_1|-|S|-2h).
\end{align*}
Note that $$|A|\geq 6|S|(h+1)-5=3|S|(2h+1)+3|S|-5\geq(|S|+2h)\cdot(2h+1).$$ We obtain that
$$
|\cA_1|\geq\frac{|A|}{m}\geq\frac{|A|}{2h+1}\geq |S|+2h,
$$
i.e., $|A\splus{S}B|\geq |A|+|B|-|S|-2$.

\medskip\noindent (ii) $|S|=h$.\medskip

Now $|\cS_i|=1$ for each $1\leq i\leq h$. We shall use another way to give the lower bound of $|A\splus{S}B|$
Since $h\geq 2$ and $|B|\geq 6h(h-1)$,
in view of (\ref{A1Bm3hSk}),  we have
\begin{align}\label{mnB5S}
|A\splus{S}B|\geq&n(|\cA_1|-1)+|B|+(m-3h)-h-|S|\notag\\
\geq&n+m+|B|-5h\geq m+n+h.
\end{align}
So we may assume that $m+n-1\leq p$. According to Lemma \ref{ABmn1G}, 
assume that $\bar{a}_1+\bar{b}_1,\ldots,\bar{a}_1+\bar{b}_n,
\bar{a}_2+\bar{b}_{\upsilon_2},\ldots,\bar{a}_m+\bar{b}_{\upsilon_m}$ are distinct elements of $\bar{A}+\bar{B}$.
For $2\leq j\leq m$, let
$$
\cT_j=\begin{cases}(b_{\upsilon_j}-a_j)+s_i+\cS_i,\qquad&\text{if }
 \bar{a}_j-\bar{b}_{\upsilon_j}=\bar{s}_i\text{ for some }1\leq i\leq h,\\
\emptyset,\qquad&\text{if }\bar{a}_j-\bar{b}_{\upsilon_j}\not\in\bar{S}.
\end{cases}
$$
Then
\begin{align}\label{ASBS1m1}
|A\splus{S}B|\geq&\sum_{k=1}^r|\cA_1\splus{\cS_{\lambda_k}^*}\cB_{\mu_k}|+
\sum_{\substack{1\leq\nu\leq m\\ \nu\not\in\{\mu_1,\ldots,\mu_r\}}}|\cA_1+\cB_\nu|+\sum_{j=2}^{m}|\cA_{j}\splus{\cT_{j}}\cB_{\upsilon_j}|\notag\\
\geq&\sum_{k=1}^r(|\cA_1|+|\cB_{\mu_k}|-3)+
\sum_{\substack{1\leq\nu\leq m\\ \nu\not\in\{\mu_1,\ldots,\mu_r\}}}(|\cA_1|+|\cB_\nu|-1)+\sum_{j=2}^{m}(|\cA_{j}|-1)\notag\\
\geq&|A|+|B|+(n-1)|\cA_1|-2h-n-(m-1)\notag\\
\geq&|A|+|B|-|S|-2+\big((n-1)|\cA_1|-2n-h+3\big).
\end{align}
Recalling that $2\leq n\leq 5h-1$, we have
\begin{align*}
(n-1)|\cA_1|-2n\geq&
(n-1)\cdot\bigg\lceil\frac{|A|}{n}\bigg\rceil-2n\\
\geq&\min\bigg\{\frac{(2-1)|A|}{2}-2\cdot 2,\ (5h-2)\cdot\bigg\lceil\frac{|A|}{5h-1}\bigg\rceil-2(5h-1)\bigg\},
\end{align*}
where $\lceil x\rceil$ denotes the least integer not less than $x$.
It is easy to verify that
$$
\frac{|A|}{2}-4\geq (3h^2-3)-4\geq h-3
$$
for $h\geq 2$,
and
$$
\frac{(5h-2)|A|}{5h-1}-2(5h-1)\geq
\frac{(5h-2)\cdot (6h^2-5)}{5h-1}-2(5h-1)\geq h-3
$$
whenever $h\geq 2$. When $h=2$, we also have
$$
(5h-2)\cdot\bigg\lceil\frac{|A|}{5h-1}\bigg\rceil-2(5h-1)
=8\cdot\bigg\lceil\frac{19}{9}\bigg\rceil-18=6>2-3.
$$
So by (\ref{ASBS1m1}), we get the desired result.

\medskip\noindent(II) $|\cA_1|=1$.\medskip

Suppose that $|\cB_1|\geq 2$. Since $|\cA_1|=\cdots|\cA_m|=1$, we have
$$
m=|A|\geq 6|S|^2-5\geq 4h.
$$
Assume that $1\leq \hat{\mu}_1,\ldots,\hat{\mu}_{\hat{r}}$ are all integers such that $\bar{a}_{\hat{\mu}_k}-\bar{b}_1=\bar{s}_{\hat{\lambda}_k}\in\bar{S}$ for $1\leq k\leq\hat{r}$.
And assume that $\bar{a}_{\hat{\gamma}_1}+\bar{b}_{\hat{\eta}_1},\ldots,\bar{a}_{\hat{\gamma}_{n-2h-1}}+\bar{b}_{\hat{\eta}_{n-2h-1}}$ are distinct elements of $(\bar{A}\splus{\bar{S}}\bar{B})\setminus\{\bar{a}_1+\bar{b}_1,\ldots,\bar{a}_m+\bar{b}_1\}$
with $\hat{\eta}_{j}\leq j+2h+1$ for $1\leq j\leq n-2h-1$.
Letting $\cS_{\hat{\lambda}_k}^*=({b}_1-{a}_{\hat{\mu}_k})+s_{\hat{\lambda}_k}+\cS_{\hat{\lambda}_k}$, we have
\begin{align}\label{ASBA11B}
|A\splus{S}B|\geq&\sum_{k=1}^{\hat{r}}|\cA_{\hat{\mu}_k}\splus{\cS_{\hat{\lambda}_k}^*}\cB_1|+
\sum_{\substack{1\leq\hat{\nu}\leq m\\ \hat{\nu}\not\in\{\hat{\mu}_1,\ldots,\hat{\mu}_{\hat{r}}\}}}|\cA_{\hat{\nu}}+\cB_1|+\sum_{j=1}^{n-2h-1}|\cA_{\hat{\gamma}_j}+\cB_{\hat{\eta}_j}|\notag\\
\geq&\sum_{k=1}^{\hat{r}}(|\cB_1|-|\cS_{\hat{\lambda}_k}^*|)+
\sum_{\substack{1\leq\hat{\nu}\leq m\\ \hat{\nu}\not\in\{\hat{\mu}_1,\ldots,\hat{\mu}_{\hat{r}}\}}}|\cB_1|+\sum_{j=1}^{n-2h-1}|\cB_{\hat{\eta}_j}|\notag\\
\geq&(2h+1)|\cB_1|+\sum_{j=1}^{n-2h-1}|\cB_{\hat{\eta}_j}|-|S|\geq |B|+(m-2h-1)|\cB_1|-|S|\notag\\
\geq&|B|+|A|+2(m-2h-1)-|S|-m\geq |A|+|B|-|S|-2.
\end{align}

Finally, assume that $|\cB_1|=1$. Then  by the induction hypothesis, 
$$
|A\splus{S}B|\geq|\bar{A}\splus{\bar{S}}\bar{B}|\geq|\bar{A}|+|\bar{B}|-|\bar{S}|-2\geq|A|+|B|-|S|-2.
$$
\end{proof}

\section{Proof of Theorem \ref{MainT2}}
\setcounter{lemma}{0}\setcounter{theorem}{0}\setcounter{corollary}{0}
\setcounter{equation}{0}

By Proposition \ref{T2pa}, there is nothing to do if $G=\Z_{p^\alpha}$. Suppose that $G$ is not a cyclic group of prime power order. 
The case $|S|=1$ easily follows from (\ref{B-W}). Note that
$$
|A\splus{S}B|\geq |A|-|S|\geq 9|S|^2-6|S|-3=3(3|S|+1)(|S|-1)\geq |S|
$$
whenever $|S|\geq 2$.
We always assume that $2\leq |S|<p(G)$. 

We use an induction on $|G|$. Assume that Theorem \ref{MainT2} holds for any abelian group whose order is less than $|G|$.
For a subgroup $H$ of $G$, define
$$
\fX_{A,H}=\{a+H\in G/H:\ A\cap(a+H)\neq\emptyset\}.
$$
We claim that $$|\fX_{A,H}|\geq\sqrt{|A|}$$ for some subgroup $H\subseteq G$ of prime order.
Since $G$ is not a cyclic group of prime power order, we may write $G=K_1\oplus K_2$ where $|K_1|,|K_2|>1$. Assume that $|\fX_{A,K_1}|<\sqrt{|A|}$. By the pigeonhole principle, there exists a coset $a+K_1$ such that 
$A\cap(a+K_1)>\sqrt{A}$. Assume that $a+K_1=\{b_1,\ldots,b_k\}$. It is easy to see that $b_1+K_2,\ldots,b_k+K_2$ are distinct cosets of $K_2$. We get $|\fX_{A,K_2}|>\sqrt{|A|}$.
So $\max\{|\fX_{A,K_1}|,|\fX_{A,K_2}|\}\geq\sqrt{|A|}$. Assume that $|\fX_{A,K_1}|\geq \sqrt{|A|}$ and let $H$ be a subgroup of $K_1$ with $|H|$ is prime. Clearly we also have
$|\fX_{A,H}|\geq\sqrt{|A|}$. Since
$$\min\{|A|,|B|\}\geq 9|S|^2-5|S|-3,$$
we may assume that.
$$
\max\{|\fX_{A,H}|,|\fX_{B,H}|\}\geq\sqrt{9|S|^2-5|S|-3}.
$$

Assume that $$\bar{A}=\{\bar{a}_1,\ldots,\bar{a}_m\},\qquad \bar{B}=\{\bar{b}_1,\ldots,\bar{b}_n\},\qquad\bar{S}=\{\bar{s}_1,\ldots,\bar{s}_h\},$$
where $\bar{a}=a+H$ and $\bar{A}=\fX_{A,H}$.
Further, without loss of generality, assume that $n\geq m$. According to our choice of $H$, we know that $n\geq\sqrt{9|S|^2-5|S|-3}$.
Write
$$
A=\bigcup_{i=1}^m(a_i+\cA_i),\qquad B=\bigcup_{i=1}^n(b_i+\cB_i),\qquad
S=\bigcup_{i=1}^h(s_i+\cS_i),
$$
where
$$
|\cA_1|\geq|\cA_2|\geq\cdots\geq|\cA_m|,\qquad
|\cB_1|\geq|\cB_2|\geq\cdots\geq|\cB_n|.
$$
If $m=n=1$, then for some $1\leq i\leq h$, $$|A\splus{S} B|\geq |\cA_1|+|\cB_1|-|\cS_{i}|-2\geq|A|+|B|-|S|-2.$$ So below assume that either $m\geq 2$ or $m<n$.

\medskip\noindent(I) $|\cA_1|\geq 2$.\medskip

We have
\begin{align}\label{}
|A\splus{S}B|\geq&\sum_{k=1}^r|\cA_1\splus{\cS_{\lambda_k}^*}\cB_{\mu_k}|+
\sum_{\substack{1\leq\nu\leq m\\ \nu\not\in\{\mu_1,\ldots,\mu_r\}}}|\cA_1+\cB_\nu|+\sum_{j=1}^{\tau}|\cA_{\gamma_j}+\cB_{\eta_j}|\notag\\
\geq&\sum_{k=1}^r(|\cA_1|+|\cB_{\mu_k}|-|\cS_{\lambda_k}|-2)+
\sum_{\substack{1\leq\nu\leq m\\ \nu\not\in\{\mu_1,\ldots,\mu_r\}}}(|\cA_1|+|\cB_\nu|-1)+\sum_{j=1}^{\tau}|\cA_{\gamma_j}|\notag\\
\geq&n|\cA_1|+|B|+\sum_{j=1}^{\tau}|\cA_{\gamma_j}|-\sum_{k=1}^r(|\cS_{k}|+1)-n,
\end{align}
where we have assumed that those
$$
|\cA_1\splus{\cS_{\lambda_k}^*}\cB_{\mu_k}|,\ |\cA_1+\cB_\nu|,\
|\cA_{\gamma_j}+\cB_{\eta_j}|<p(G).
$$
Suppose that $m\geq 3h+1$. Then
\begin{equation}\label{ASBn3hA1}
|A\splus{S}B|\geq |A|+|B|-|S|-2+\big((n-3h)|\cA_1|-n-h+2\big),
\end{equation}
since $\tau\geq m-3h$ now. 
And if $m\leq 3h$, we also have
\begin{align}\label{ASBnmA1}
|A\splus{S}B|\geq&|A|+|B|-|S|-2+\big((n-m)|\cA_1|-n-h+2\big).
\end{align}

If $n\geq 7h-2$, then 
$$
|A\splus{S}B|\geq |A|+|B|+2(n-3h)-|S|-h-n\geq |A|+|B|-|S|-2.
$$
Below assume that $n\leq 7h-3$. 

\medskip\noindent(i) $|S|\geq h+1$. \medskip

By (\ref{ASBn3hA1}),
we have
\begin{align*}
|A\splus{S}B|\geq |A|+|B|-|S|-2+\bigg(\frac{n-3h}{n}\cdot|A|-n-h+2\bigg).
\end{align*}
Clearly the function $(x-3h)|A|/x+x$ is increasing on $(0,\sqrt{3h|A|}]$ and is 
decreasing on $[\sqrt{3h|A|},+\infty)$. 
Note that $$|A|\geq 9|S|^2-5|S|-3\geq 9h^2+13h+1.$$
It is easy to check that
$$
\sqrt{3h|A|}\geq\sqrt{3h(9h^2+13h+1)}\geq\max\{\sqrt{9h^2+13h+1},\ 7h-3\}.
$$
If $h\geq 2$, then $\sqrt{9h^2+13h+1}\leq 7h-3$. So
\begin{align*}
\frac{n-3h}{n}\cdot|A|-n\geq&(\sqrt{9h^2+13h+1}-3h)\cdot\frac{|A|}{\sqrt{9h^2+13h+1}}-\sqrt{9h^2+13h+1}\\
\geq&(\sqrt{9h^2+13h+1}-3h)\cdot\sqrt{9h^2+13h+1}-\sqrt{9h^2+13h+1}\\
=&h-2+\big(9h^2+12h+3-(3h+1)\sqrt{9h^2+13h+1}\big)\geq h-2,
\end{align*}
where the last step follows from
$$\frac{9h^2+12h+3}{3h+1}=3h+3>\sqrt{9h^2+13h+1}.
$$
And if $h=1$, then
\begin{align*}
\frac{n-3h}{n}\cdot|A|-n\geq(7h-6)\cdot\frac{|A|}{7h-3}-(7h-3)\geq1\cdot\frac{23}{4}-4>h-2.
\end{align*}
Thus we get (\ref{ABS1S2}).

\medskip\noindent(ii) $|S|=h$. \medskip

By (\ref{mnB5S}), we may assume that $m+n-1\leq p(G)$.
In view of (\ref{ASBS1m1}), we get
\begin{align*}
|A\splus{S}B|
\geq&|A|+|B|+(n-1)|\cA_1|-2h-n-(m-1)\\
\geq&|A|+|B|-|S|-2+\bigg((n-1)\cdot\bigg\lceil\frac{|A|}{n}\bigg\rceil-n-h+3\bigg).
\end{align*}
For $2\leq n\leq 7h-3$, we have
$$
(n-1)\cdot\bigg\lceil\frac{|A|}{n}\bigg\rceil-n\geq\min\bigg\{\bigg\lceil\frac{|A|}{2}\bigg\rceil-2,\ (7h-4)\cdot\bigg\lceil\frac{|A|}{7h-3}\bigg\rceil-(7h-3)\bigg\}.
$$
Since $|A|\geq 9h^2-5h-3$, it is not difficult to verify that
$|A|/2-2\geq h-3$
for $h\geq 2$, and
$$
\frac{(7h-4)|A|}{7h-3}-(7h-3)\geq h-3
$$
for $h\geq 3$. And if $h=2$, then 
$$
(7h-4)\cdot\bigg\lceil\frac{|A|}{7h-3}\bigg\rceil-(7h-3)\geq 10\cdot\bigg\lceil\frac{23}{11}\bigg\rceil-11=9\geq h-3.
$$
So we have $|A\splus{S}B|\geq|A|+|B|-|S|-2$.

\medskip\noindent(II) $|\cA_1|=1$.\medskip

If $|\cA_1|=|\cB_1|=1$, then our assertion immediately follows from the induction hypothesis on $G/H$.
Suppose that $|\cA_1|=1$ and $|\cB_1|\geq 2$.
Then 
$$
m=|A|\geq 9|S|^2-5|S|-3\geq 6h-2.
$$
Assume that $\bar{a}_{\hat{\mu}_k}-\bar{b}_1=\bar{s}_{\hat{\lambda}_k}\in\bar{S}$ for $1\leq k\leq\hat{r}$, and
$\bar{a}_{\hat{\gamma}_1}+\bar{b}_{\hat{\eta}_1},\ldots,\bar{a}_{\hat{\gamma}_{n-3h}}+\bar{b}_{\hat{\eta}_{n-3h}}$ are distinct elements of $(\bar{A}\splus{\bar{S}}\bar{B})\setminus\{\bar{A}+\bar{b}_1\}$
with $\hat{\eta}_{j}\leq j+3h$ for $1\leq j\leq n-3h$.
Let $\cS_{\hat{\lambda}_k}^*=({b}_1-{a}_{\hat{\mu}_k})+s_{\hat{\lambda}_k}+\cS_{\hat{\lambda}_k}$. We obtain that
\begin{align}\label{}
|A\splus{S}B|\geq&\sum_{k=1}^{\hat{r}}|\cA_{\hat{\mu}_k}\splus{\cS_{\hat{\lambda}_k}^*}\cB_1|+
\sum_{\substack{1\leq\hat{\nu}\leq m\\ \hat{\nu}\not\in\{\hat{\mu}_1,\ldots,\hat{\mu}_{\hat{r}}\}}}|\cA_{\hat{\nu}}+\cB_1|+\sum_{j=1}^{n-3h}|\cA_{\hat{\gamma}_j}+\cB_{\hat{\eta}_j}|\notag\\
\geq&\sum_{k=1}^{\hat{r}}(|\cB_1|-|\cS_{\hat{\lambda}_k}^*|)+
\sum_{\substack{1\leq\hat{\nu}\leq m\\ \hat{\nu}\not\in\{\hat{\mu}_1,\ldots,\hat{\mu}_{\hat{r}}\}}}|\cB_1|+\sum_{j=1}^{n-3h}|\cB_{\hat{\eta}_j}|\notag\\
\geq&3h|\cB_1|+\sum_{j=1}^{n-3h}|\cB_{\hat{\eta}_j}|-|S|\geq |B|+(m-3h)|\cB_1|-|S|\notag\\
\geq&|B|+|A|+2(m-3h)-|S|-m\geq |A|+|B|-|S|-2.
\end{align}

\end{document}